\documentclass[sn-basic,Numbered]{sn-jnl}
\usepackage[utf8]{inputenc}
\usepackage{graphicx}%
\usepackage{multirow}%
\usepackage{amsmath,amssymb,amsfonts}%
\usepackage{amsthm}%
\usepackage{mathrsfs}%
\usepackage[title]{appendix}%
\usepackage{xcolor}%
\usepackage{textcomp}%
\usepackage{manyfoot}%
\usepackage{booktabs}%
\usepackage{algorithm}%
\usepackage{algorithmicx}%
\usepackage{algpseudocode}%
\usepackage{listings}%

\usepackage{tikz}
\usepackage{graphicx, caption,subcaption}
\usepackage[shortlabels]{enumitem}

\newcommand{\prob}{\mathbb{P}}

\newcommand{\expec}{\mathbb{E}}
\newcommand{\Exp}[1]{\expec\left[#1\right]}

\newcommand{\Var}[1]{\textup{Var}\left(#1\right)}

\newcommand{\plim}{\xrightarrow[n\to \infty]{\mathbb{P}}}
\newcommand{\dlim}{\xrightarrow[n\to \infty]{d}}

\newcommand{\bigO}[1]{O\left(#1\right)}
\newcommand{\bigOp}[1]{O_{\prob}\left(#1\right)}

\newcommand{\op}{o_{\prob}}

\newtheorem{theorem}{Theorem}[section]

\newtheorem{remark}{Remark}[section]
\newtheorem{lemma}[theorem]{Lemma}

\newtheorem{corollary}[theorem]{Corollary}

\begin{document}
\title[On the distances within cliques in a soft random geometric graph]{On the distances within cliques in a soft random geometric graph}
%\title{On the distances within cliques in a soft random geometric graph}
%\author[1]{Ercan S\"onmez}
%\author[2]{Clara Stegehuis}
\author[1]{\fnm{Ercan} \sur{S\"onmez}}\email{ercan.soenmez@mathematik.uni-stuttgart.de}
\author[2]{\fnm{Clara} \sur{Stegehuis}}\email{c.stegehuis@utwente.nl}
\affil[1]{\orgname{University of Stuttgart}}
\affil[2]{\orgname{University of Twente}}
\date{\today}

%\begin{abstract}
\abstract{
    We study the distances of edges within cliques in a soft random geometric graph on a torus, where the vertices are points of a homogeneous Poisson point process, and far-away points are less likely to be connected than nearby points. We obtain the scaling of the maximal distance between any two points within a clique of size $k$. Moreover, we show that asymptotically in all cliques with large distances, there is only one remote point and all other points are nearby. Furthermore, we prove that a re-scaled version of the maximal $k$-clique distance converges in distribution to a Fr\'echet distribution. Thereby, we describe the order of magnitude according to which the largest distance between two points in a clique decreases with the clique size. 
    }
    %Furthermore, we show that asymptotically all cliques with large edges are formed with $k-1$ long edges, and $(k-1)^2/2$ close by edges. 
%\end{abstract}

\maketitle

\section{Introduction}
While networks appear in many different applications, many real-world networks can be thought of as being embedded in some geometric space. In social networks, for example, the geometric space may contain information on people's locations and interests, while in networks of chemical reactions, the underlying geometric space may contain information on the properties of the involved molecules. Usually, nearby vertices in these geometric spaces are then more likely to be connected than far away vertices. For this reason, many random graph models that generate geometric networks have been developed. Usually, vertices are then embedded in a continuous, $d$-dimensional space, and nearby vertices are more likely to connect than further apart ones~\cite{bringmann2015,coppersmith2002diameter,penrose2003}.

In general, cliques are an important indicator of structural network properties of a network. Large cliques, or clique-like structures, indicate the tendency of a network to cluster into groups. Small cliques on the other hand are related to the frequently analyzed clustering coefficient of a network.
Therefore, cliques have been an object of intensive studies in many random graph models~\cite{bianconi2006number,Blasius2017,devroye2011,janson2010,janssen2019,michielan2021,penrose2003}. Often, these works focus on the total number of cliques, or the largest clique that is present in the network. Comparatively, little attention has been devoted to the properties of these cliques instead. Here, we focus on the structures of the cliques and investigate \emph{which types of edges are part of cliques}. We investigate this problem for the soft random geometric graph~\cite{penrose2016}, a random graph model that ensures that nearby vertices are more likely to connect to each other than vertices that are further apart, but also allows for longer connections. In this soft version, the random geometric graph can be seen as the grand canonical ensemble that maximizes ensemble entropy under the constraints that the average number
of particles (edges) and the average energy that depends on the distance between nodes are fixed to given values~\cite{boguna2020}. As such, properties of the soft random geometric graph and other versions of it that add more constraints on the model~\cite{boguna2020,bringmann2015} have been an object of intense study~\cite{bringmann2015,boguna2020,rs1,rs2,ostilli2015}. One particular question of interest is in understanding how networks respond to perturbations or failures, see \cite{callaway, cohen}. The longest edge might be a critical factor in determining the robustness of the network structure. Thus, the study of edge-lengths naturally emerges as crucially interesting, offering insights into the probabilistic aspects of network structure and potentially playing a pivotal role in understanding the resilience and response of complex systems.

Intuitively, in geometric models, cliques are formed between nearby vertices, as close by vertices are more likely to connect. However, in the soft random geometric graph, some long edges are present as well. The extreme value properties of these edge-lengths have been investigated in \cite{rs1,rs2}, revealing that the longest edge scales polynomially in the length in a given observation window of the model. The presence of these long edges indicates that some cliques may still contain longer edges as well. Our work therefore focuses on the question: what is the longest distance between two points in any $k$-clique? And how many cliques contain long edges?

 We investigate the number of cliques with at least one edge of length larger than some given threshold $r_n$ in the soft random geometric graph on a $d$-dimensional torus of radius $n$, and we study their properties as $n \to \infty$. We prove a scaling limit for the number of $k$-cliques with at least one edge of length exceeding $r_n$, and prove a localization phenomenon, meaning that if there is at least one long edge in a given clique, then asymptotically there are $k-2$ long edges with one equal endpoint, while all other edges of the clique are short. We also show that the longest edge over all cliques of size $k$ decreases in length with growing $k$. In particular, its length scales as $n^{d/(d-(k-1)\alpha)}$, where $\alpha$ is a parameter of the random graph model that controls the likelihood of long edges. Furthermore, we show that its re-scaled version converges to a Fr\'echet distribution. Interestingly, for all fixed $k$, the total number of $k$-cliques scales as $n^d$ in this model. Thus, while the total number of cliques scales similarly in the network size for all fixed $k$, their extreme behavior is remarkably different. 

\paragraph{Organization of the paper.}
In Section \ref{sec:model} we give a mathematical formulation of the model we investigate. In Section \ref{sec:results} we state the main results of this article. These include the statement of a localization phenomenon, i.e. asymptotically, there is only one remote vertex in cliques with sufficiently large distances, while all other vertices are close to each other, and a scaling limit for the total number of cliques with at least one edge of length $r_n$. Moreover, we provide a characterization of the largest distance within cliques of arbitrary size. In particular, we find an extreme value behavior by showing that a re-scaled version of the maximal clique distance converges in distribution to a Fr\'echet distribution. Thereby, we describe the order of magnitude according to which the largest distance decreases with the clique size. The results are accompanied by simulations in order to support our findings. The remaining Sections \ref{sec:proofs}, \ref{pr3} and \ref{pr2} are devoted to the proofs of the individual results. In order to facilitate readability, we provide an index of notation for some objects defined in this paper at the end of the paper.

\subsection{Model}\label{sec:model}
We now recall the definition of a soft random geometric graph, which dates back to \cite{penrose2016}. It is a finite version of the random connection model, a standard model of continuum percolation (see \cite{MR,penrose1991}), and can be defined on more general bounded domains (see \cite{giles}). 

Throughout this paper, let $d, n\in\mathbb{N}$, and let $\mathcal{P}$ be a homogeneous Poisson point process on $\mathbb{R}^d$ with unit intensity. We will consider $\mathcal{P}$ on the $d$-dimensional torus of radius $n$ denoted by $\mathcal{T}_n^d$ and we view $\mathcal{P}$ as a random countable subset of $\mathcal{T}_n^d$. Given the vertex set $\mathcal{P}$ the edge-set is constructed as follows. 
%Let $g\colon \mathbb{R}^d \times \mathbb{R}^d \to [0,1]$ be a measurable function. We assume that $g$ is symmetric, i.e. $g(x,y) = g(y,x)$ for all $x,y \in \mathbb{R}^d$. 
Each pair of points $\mathbf{x},\mathbf{y} \in \mathcal{P}$ is connected by an edge $\{\mathbf{x},\mathbf{y}\}$ with probability $g(\mathbf{x},\mathbf{y})$ independently of all other pairs of points, where we assume that $g$ is given by
\begin{equation}\label{cf}
	 g(\mathbf{x},\mathbf{y})=1-\exp(-|\mathbf{x}-\mathbf{y}|_T^{-\alpha}), \quad \mathbf{x},\mathbf{y} \in \mathbb{R}^d,
\end{equation}
for $\alpha\in (d,\infty)$, where 
%\begin{equation}
%    |x-y|_T=\max_{1\leq i\leq d}\min(|x_i-y_i|, n-|x_i-y_i|).
%\end{equation}
\begin{equation}
    |\mathbf{x}-\mathbf{y}|_T=\sqrt{\sum_{i=1}^d\min(|x_i-y_i|, n-|x_i-y_i|)^2}.
\end{equation}
In particular $g$ has unbounded support and polynomial decay as $|\mathbf{x}-\mathbf{y}|_T \to \infty$. For $x_i\in\mathbb{R}$, we will also denote $|x_i|_T=\min(|x_i|, n-|x_i|)$.

\subsection{Notation} We end this section by introducing some notation. For $d\in\mathbb{N}$ we write $[d]=\{1, \ldots, d\}$. For two functions $f,g$ we write $f(n) = o(g(n))$ if $\lim_{n \to \infty} f(n)/g(n)=0$, and $f(n) = O(g(n))$ if $\limsup_{n \to \infty} f(n)/g(n) < \infty$. For a sequence of random variables $X_n$ and a deterministic sequence $a_n$, $n\in \mathbb{N}$, we write $X_n = O_\mathbb{P}(a(n))$ if for every $\varepsilon>0$ there exists $M>0$, $N\in \mathbb{N}$ with $\mathbb{P} ( |X_n|/a(n) \geq M) < \varepsilon$, for all $n>N$. Furthermore, we write $X_n = o_\mathbb{P}(a(n))$ if for every $\varepsilon>0$ it holds that $\mathbb{P} ( |X_n|/a(n) \geq \varepsilon) \to 0$, as $n\to \infty$. Moreover, we write $\plim$ for convergence in probability, and $\dlim$ for convergence in distribution.

\section{Main results}\label{sec:results}

Throughout this paper, we let $k \in \mathbb{N}$ with $k\geq 3$. We will consider cliques of size $k$ ($k$-cliques) in the soft random geometric graph, and focus on their geometric structure. In particular, we focus on the number of cliques with at least one edge of length at least $r_n$. We also state a localization phenomenon. That is, with high probability, cliques with long edges contain only one remote vertex, and all other vertices of the clique are nearby.

To be more precise, let $W_n^k(r_n)$ denote the number of $k$-cliques with at least one edge of length at least $r_n$, and $r_n$ is a given sequence with $r_n \to \infty$, as $n \to \infty$. Let $W_n^k(r_n,\varepsilon)$ denote the number of $k$-cliques with exactly $k-1$ edges of length at least $r_n$, and all other edges of length at most $1/\varepsilon$, for some $\varepsilon>0$. Then, the following theorem shows that these types of cliques with exactly $k-1$ long edges are asymptotically all cliques with long edges:
\begin{theorem}[Number of cliques with long edges] \label{thm:localization_long_cliques}
    When $1\ll r_n\ll n^{d/((k-1)\alpha-d)}$, then for all $\varepsilon_n$ such that $\lim_{n\to\infty}\varepsilon_n=0$ and $\varepsilon_n\geq \log(n)^{-1}$,
    \begin{equation*}
        \frac{W_n^k(r_n,\varepsilon_n)}{W_n^k(r_n)} \plim 1.
    \end{equation*}
    Furthermore,
    \begin{equation*}
        \frac{W_n^k(r_n)}{n^{d}r_n^{d-(k-1)\alpha}} \plim \frac{dC_d\pi^{d/2} M_k}{2\Gamma(1+\frac{d}{2})((k-1)\alpha-d)},
    \end{equation*}
    where 
\begin{equation} \label{mk}
    M_k = 2\int_{[-\infty,\infty]^d}\dots\int_{[-\infty,\infty]^d}\prod_{i=1}^{k-2}g(\mathbf{x_i},\mathbf{0})\prod_{1\leq u<v\leq k-2}g(\mathbf{x_u},\mathbf{x_v})d\mathbf{x_1}\dots d\mathbf{x_{k-2}}<\infty.
\end{equation}
and $C_d $ is the constant such that the volume of the $d$-dimensional torus is $C_dn^d.$
In particular, for triangles (when $k=3$),
\begin{equation*}
    M_3 = \frac{2\pi^{d/2}\Gamma(1-\frac{d}{\alpha})}{\Gamma(1+d/2)}.
\end{equation*}
%and
%$$ \frac{W_n^3(r_n)}{n^{d}r_n^{d-2\alpha}} \plim \frac{2^{d}C_3M_3}{2\alpha-d}.$$
Moreover, 
\begin{equation*}
    \frac{W^k_n(r_n) - \mathbb{E}[W^k_n(r_n)] }{\sqrt{\mathbb{E}[W^k_n(r_n)]}}  \xrightarrow[n\to \infty]{d} Z,
\end{equation*}
where $Z$ is a random variable with standard normal distribution. \newline
\end{theorem}
%\cs{Maybe we can remove the convergence in probability part here, and only write the convergence of the expectation instead? As the CLT is stronger ?}
This theorem shows that asymptotically all cliques with long edges are formed between $k-1$ close by vertices within a radius of $1/\varepsilon_n$, and $k-1$ edges of length at least $r_n$. Furthermore, the scaling of the largest distance in a $k$-clique decreases in $k$, as $\alpha>d$. 

We now focus on the boundary case of Theorem~\ref{thm:localization_long_cliques}, where $r_n\sim n^{d/((k-1)\alpha-d)}$, and show that this corresponds to the explicit asymptotics of the longest distance found in any $k$-clique. The following theorem shows that, for this choice of the sequence $r_n$, we now get a convergence in distribution to a Poisson distribution:

%\cs{Should this be a lemma or thm?}
\begin{theorem}[Large distances in $k$-cliques] \label{lem: wnc}
Let $$r_n = \Bigg(\frac{2\Gamma(1+\frac{d}{2})((k-1)\alpha-d)}{dC_d\pi^{d/2} M_k}\Bigg)^{\frac{1}{d-(k-1)\alpha}} r n^{\frac{d}{(k-1)\alpha-d}}, \quad r \in (0, \infty),$$
where the constants $C_d,M_k$ are as in Theorem \ref{thm:localization_long_cliques}. Then
    \begin{equation*}
    W^k_n(r_n)  \xrightarrow[n\to \infty]{d} W^k,
\end{equation*}
where $W^k$ denotes a random variable with Poisson distribution and mean $r^{d-(k-1)\alpha}$.
\end{theorem}

\begin{remark}\label{ustat}
    The quantity $W^k_n(r_n)$ can be considered to be a U-statistic of order $k$, for which there are powerful results regarding Poisson approximation, see for example \cite[Theorem 7.1]{decreusefond2016}. However, the method that we will use in this paper relies on the approach presented in \cite{Penrose}, as it will give us stronger bounds that will also enable us to prove the central limit theorem in Theorem \ref{thm:localization_long_cliques}, see also Remark \ref{ustat2} below for a more detailed discussion.
\end{remark}

A consequence of Theorem \ref{lem: wnc} is the following corollary, which gives the explicit asymptotics of the longest distance found in $k$-cliques and reveals its extreme value behavior.

\begin{corollary}[Maximal distance within a clique] \label{frechet}
Let $e_n^{k,*}$ be the largest distance within a clique of size $k$, i.e. the largest edge-length in such a clique. Then,
\begin{equation*}
    \Bigg( \frac{2\Gamma(1+\frac{d}{2})((k-1)\alpha-d)}{dC_d\pi^{d/2} M_k} \Bigg)^{\frac{1}{(k-1)\alpha-d}} n^{\frac{d}{d-(k-1)\alpha}} e_n^{k,*} \xrightarrow[n\to \infty]{d} \Phi_{(k-1)\alpha-d},
\end{equation*}
where $\Phi_{(k-1)\alpha-d}$ denotes a Fr\'echet distribution with parameter $(k-1)\alpha-d$, and the constants $C_k, M_k$ are as in Theorem \ref{thm:localization_long_cliques}.
\end{corollary}

In \cite{rs1,rs2} the scaling of the length of the overall longest edge (i.e. not necessarily in a clique) was found to be $n^{\frac{d}{\alpha-d}}$, up to a constant, which would correspond to the same scaling if we plugged in $k=2$ in Corollary \ref{frechet}. This would indicate that the overall largest edge-length is typically not part of a clique.

Finally, we focus on the overall distances in $k$-cliques. Let $K_k(\varepsilon_n)$ denote the number of $k$-cliques where all vertices have interdistances at most $1/\varepsilon_n$, and $K_k$ the total number of $k$-cliques.
\begin{theorem}\label{thm:cliques_total}
    For all $\varepsilon_n$ such that $\lim_{n\to\infty}\varepsilon_n=0$ and $\varepsilon_n\geq \log(n)^{-1}$,
    \begin{equation*}
        \frac{K_k(\varepsilon_n)}{K_k} \plim 1.
    \end{equation*}
\end{theorem}

\paragraph{Simulations.}
To illustrate our results, we provide some simulations, where we simulate a soft random geometric graph on a (one-dimensional) torus and compute the empirical distribution function of the longest edge among all triangles. We plot the empirical cumulative distribution function (see Figure \ref{fig:simf2}) derived from repeated simulations of the length of the longest edge among all triangles (in dimension $d=1$ for the value $\alpha=4$), normalized by %\cs{Where does this normalization come from? That is just the constant in Lemma 3.2}
$$ \Big[\frac{\pi \Gamma(\frac{3}{4})}{7\Gamma(\frac{3}{2})^2} \Big]^{\frac{1}{7}} n^{\frac{1}{7}}.$$ 
According to Corollary \ref{frechet}, this normalized length approximates a Fr\'echet distribution with parameter $7$, which is also supported by the simulations.
%\begin{figure}\label{f1}
%	\begin{subfigure}[l]{0.4\textwidth}
	%	\input{sim3}
	%	\caption{$n=2$, $\alpha =4$}
%	\end{subfigure}
%    \begin{subfigure}[r]{0.4\textwidth}
	%	\input{sim4}
	%	\caption[c]{$n=5$, $\alpha =3$}
	%\end{subfigure}
%\caption{Visualizations of the random graph model}
	%\label{fig:simf}
%\end{figure}

%\newline

\begin{figure}[tbp]
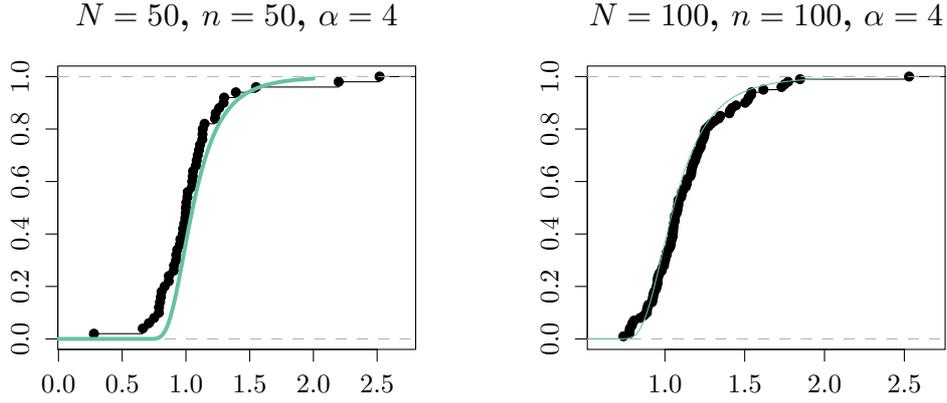

	\begin{minipage}[t]{0.4\textwidth}
 \begingroup
\tikzset{every picture/.style={scale=0.75}}%

		\input{sim5}
  \endgroup
		%\caption[c]{$N=50$, $n=50$, $\alpha =4$}
	\end{minipage}
 \hspace{1.5 cm}
	\begin{minipage}[t]{0.4\textwidth}
  \begingroup
\tikzset{every picture/.style={scale=0.75}}%
		\input{sim6}
  \endgroup
		%\caption{$N=100$, $n=100$, $\alpha =4$}
	\end{minipage}
 \caption{Empirical distribution function of the normalized largest edge-length among all triangles derived from $N$ simulations (points) compared with the distribution function of the Fr\'echet distribution from Corollary~\ref{frechet} (green line).}
 	\label{fig:simf2}
\end{figure}

\section{Proof of Theorem~\ref{thm:localization_long_cliques}}\label{sec:proofs}
In this section, we prove Theorem~\ref{thm:localization_long_cliques}. We first focus on triangles, and then show how these results extend to the $k$-clique setting.

\subsection{Special case: triangles}
Recall that by assumption the intensity of the Poisson point process equals $1$. 
Let $\mathbf{0}$ denote the zero vector and $\mathbf{r_n}=[r_n,0,\dots,0]\in \mathbb{R}^d$. Let $\Exp{T_n(r_n)}$ be the average number of triangles that a given edge of length $r_n$ is part of. We compute
\begin{equation}\label{eq:Etriangn}
    \Exp{T_n(r_n)}= \int_{\mathcal{T}_n^d}g(\mathbf{x},\mathbf{0})g(x,\mathbf{r_n})d\mathbf{x} .
\end{equation}
Then
\begin{equation*}
    \Exp{T_n(r_n)}= \int_{\mathcal{T}_n^d}(1-\exp(-|\mathbf{x}|_T^{-\alpha}))(1-\exp(-|\mathbf{x}-\mathbf{r}_n|_T^{-\alpha}))dx.
\end{equation*}
The following lemma then gives an explicit expression for $\Exp{T_n(r_n)}$:
\begin{lemma}\label{lem:Tn}
Let $1 \ll r_n\ll n$ and $\alpha>d$. Then,
\begin{equation*}
\frac{\Exp{T_n(r_n)}}{ r_n^{-\alpha}}= \frac{2\pi^{d/2}\Gamma(1-\frac{d}{\alpha})}{\Gamma(1+\frac{d}{2})}(1+o(1)) .
\end{equation*}
\end{lemma}
\begin{proof}
We prove the lemma by constructing a matching upper and lower bound for the expectation.

\emph{Lower bound.}
Let $\mathcal{B}_{0,1/\varepsilon_n}$ be the $d$-dimensional ball of radius $1/\varepsilon_n$ around 0, and $\mathcal{B}_{\mathbf{r_n},1/\varepsilon_n}$ the ball of radius $1/\varepsilon_n$ around $\mathbf{r_n}$, where $\varepsilon_n\to 0$ as $n\to\infty$.
To lower bound~\eqref{eq:Etriangn} we focus on the contribution from the integral on $\mathcal{B}_{0,1/\varepsilon_n}$ and $\mathcal{B}_{\mathbf{r_n},1/\varepsilon_n}$, which are the same by symmetry.

For $0\ll r_n\ll n$, 
\begin{align}
& \int_{\mathcal{B}_{0,1/\varepsilon_n}}(1-\exp(-|\mathbf{x}|_{T}^{-\alpha}))(1-\exp(-|\mathbf{x}-\mathbf{r}_n|_{T}^{-\alpha}))d\mathbf{x}\nonumber\\
    %& =  \int_{-1/\varepsilon_n}^{1/\varepsilon_n}\dots \int_{-1/\varepsilon_n}^{1/\varepsilon_n}(1-\exp(-|x|_{2,T}^{-\alpha}))(1-\exp(-|x-\mathbf{r}_n|_{2,T}^{-\alpha}))dx_d\dots dx_1\nonumber\\
    & =  r_n^{-\alpha}(1+o(1)) \int_{\mathcal{B}_{0,1/\varepsilon_n}}(1-\exp(-|\mathbf{x}|_{T}^{-\alpha}))d\mathbf{x}.
\end{align}

Then, switching to spherical coordinates yields
\begin{align}
    &  \int_{\mathcal{B}_{0,1/\varepsilon_n}}(1-\exp(-|\mathbf{x}|_{T}^{-\alpha}))d\mathbf{x}\nonumber\\
    & = \int_0^{1/\varepsilon_n}\int_{0}^{2\pi}\int_0^\pi\dots \int_0^\pi \rho^{d-1}(1-\exp(-\rho^{-\alpha}))\sin(\phi_1)^{d-2}\dots \sin(\phi_{d-2})d\phi_1 \dots d\phi_{d-1} d\rho\nonumber\\
    & = \frac{2\pi^{d/2}}{\Gamma(\frac{d}{2})} \int_0^{1/\varepsilon_n} \rho^{d-1}(1-\exp(-\rho^{-\alpha})) d\rho.
    %\nonumber\\
 %   & = \frac{2^{d}\pi}{d} r_n^{-\alpha}\Gamma\Big(1-\frac{d}{\alpha}\Big)(1+o(1)).
\end{align}

% \begin{align}\label{eq:balltn}
% & \int_{\mathcal{B}_{0,1/\varepsilon_n}}(1-\exp(-|\mathbf{x}|_T^{-\alpha}))(1-\exp(-|\mathbf{x}-\mathbf{r}_n|_T^{-\alpha}))dx_d\dots dx_1\nonumber\\
%     & =  \int_{-1/\varepsilon_n}^{1/\varepsilon_n}\dots \int_{-1/\varepsilon_n}^{1/\varepsilon_n}(1-\exp(-|\mathbf{x}|_T^{-\alpha}))(1-\exp(-|\mathbf{x}-\mathbf{r}_n|_T^{-\alpha}))dx_d\dots dx_1\nonumber\\
%     & = 2^d d! r_n^{-\alpha}(1+o(1)) \int_0^{1/\varepsilon_n}\int_{0}^{x_1}\dots \int_0^{x_{d-1}}(1-\exp(-x_1^{-\alpha}))dx_d\dots dx_1\nonumber\\
%     & = 2^dd!  r_n^{-\alpha}(1+o(1))\int_0^{1/\varepsilon_n}\frac{1}{(d-1)!}x_1^{d-1}(1-\exp(-x_1^{-\alpha}))dx_1,
% \end{align}
% where the factor $2^d$ comes from the reduction to the $[0,1/\varepsilon_n]^d$ box, and the factor $d!$ arises from assuming w.l.o.g. that $x_1>x_2\dots>x_d$.
Thus,
\begin{align*}
    \lim_{n\to\infty}\frac{\Exp{T_n(r_n)}}{r_n^{-\alpha}}\geq    \frac{4\pi^{d/2}}{\Gamma(\frac{d}{2})} \int_0^{\infty}\rho^{d-1}(1-\exp(-\rho^{-\alpha}))d\rho,
\end{align*}
where the extra factor 2 arises from the fact that the integral over $\mathcal{B}_{0,1/\varepsilon_n}$ and $\mathcal{B}_{\mathbf{r_n},1/\varepsilon_n}$ are the same by symmetry.
Partial integration and a substitution of $u=x^{-\alpha}$ gives
\begin{align*}
    \int_0^{\infty}x^{d-1}(1-\exp(-x^{-\alpha}))dx 
    & = \frac{\alpha}{d} \int_0^{\infty}x^{d}x^{-\alpha-1}\exp(-x^{-\alpha})dx + \frac{1}{d}x^d(1-\exp(-x^{-\alpha}))\Big|_0^\infty\nonumber\\
    & = \frac{1}{d} \int_0^{\infty}u^{-d/\alpha}\exp(-u)du =\frac{1}{d} \Gamma(1-d/\alpha),
\end{align*}
where $\Gamma$ denotes the gamma function. Taking the limit of $n\to\infty$ yields
\begin{align}\label{eq:tnlb}
    \lim_{n\to\infty}\frac{\Exp{T_n(r_n)}}{ r_n^{-\alpha}}\geq   \frac{4\pi^{d/2}}{d\Gamma(\frac{d}{2})}\Gamma(1-\frac{d}{\alpha}) = \frac{2\pi^{d/2}\Gamma(1-\frac{d}{\alpha})}{\Gamma(1+\frac{d}{2})}.
\end{align}

\emph{Upper bound.}
%Use the same methodology, and when $x\ll r_n$: apply same bounds. when $x\ll r_n$ $|r_n-x|\approx x$ and when $x\approx r_n$ flip coordinates?
For an upper bound, we show that the contribution to the expected value on $\mathcal{B}_{0,1/\varepsilon_n}$ and $\mathcal{B}_{\mathbf{r_n},1/\varepsilon_n}$ dominates, and that all other contributions to the expectation are asymptotically smaller. We therefore show that the integral of~\eqref{eq:Etriangn} over areas where $\mathbf{x}\notin\mathcal{B}_{0,1/\varepsilon_n}\cup\mathcal{B}_{\mathbf{r_n},1/\varepsilon_n}$ is small. 

We first investigate the region where $|\mathbf{x}-\mathbf{r_n}|_T\geq r_n/2$. This yields
\begin{align}\label{eq:boundballfar}
  &  \int_{\mathcal{T}_n^d\setminus(\mathcal{B}_{0,1/\varepsilon_n}\cup\mathcal{B}_{\mathbf{r_n},1/\varepsilon_n}):|\mathbf{x}-\mathbf{r_n}|_T\geq r_n/2} (1-\exp(-|\mathbf{x}|_T^{-\alpha}))(1-\exp(-|\mathbf{x}-\mathbf{r}_n|_T^{-\alpha}))d\mathbf{x}\nonumber\\
  & = r_n^{-\alpha}(1+o(1))\int_{\mathcal{T}_n^d\setminus(\mathcal{B}_{0,1/\varepsilon_n}\cup\mathcal{B}_{\mathbf{r_n},1/\varepsilon_n}):|\mathbf{x}-\mathbf{r_n}|_T\geq r_n/2} (1-\exp(-|\mathbf{x}|_T^{-\alpha}))d\mathbf{x}\nonumber\\
  & \leq r_n^{-\alpha}(1+o(1))\int_{\mathcal{T}_n^d\setminus\mathcal{B}_{0,1/\varepsilon_n}} (1-\exp(-|\mathbf{x}|_T^{-\alpha}))d\mathbf{x}.
\end{align}
Furthermore, switching to spherical coordinates yields
\begin{align}\label{eq:boundballfar2}
    &\int_{\mathcal{T}_n^d\setminus\mathcal{B}_{0,1/\varepsilon_n}} (1-\exp(-|\mathbf{x}|_T^{-\alpha}))d\mathbf{x}\nonumber\\
    & \leq \int_{1/\varepsilon_n}^\infty  \int_{0}^{2\pi}\int_0^\pi\dots \int_0^\pi \rho^{d-1}(1-\exp(-\rho^{-\alpha}))\sin(\phi_1)^{d-2}\dots \sin(\phi_{d-2})d\phi_1 \dots d\phi_{d-1} d\rho \nonumber\\
    & = \frac{2\pi^{d/2}}{\Gamma(\frac{d}{2})} \int_{1/\varepsilon_n}^\infty  \rho^{d-1}(1-\exp(-\rho^{-\alpha}))d\rho.
    %\nonumber\\
    %& \leq \tilde{C}r_n^{d-\alpha-1}
\end{align}
%for some $\tilde{C}>0$. 
Moreover, the contribution to the integral of~\eqref{eq:Etriangn} over the region where $|\mathbf{x}-\mathbf{r_n}|_T< r_n/2$ it symmetric to the contribution over the region where $|\mathbf{x}|_T< r_n/2$, which equals
\begin{align}\label{eq:boundballclose1}
  &  \int_{\mathcal{T}_n^d\setminus(\mathcal{B}_{0,1/\varepsilon_n}\cup\mathcal{B}_{\mathbf{r_n},1/\varepsilon_n}):|\mathbf{x}|_T< r_n/2} (1-\exp(-|\mathbf{x}|_T^{-\alpha}))(1-\exp(-|\mathbf{x}-\mathbf{r}_n|_T^{-\alpha}))d\mathbf{x}\nonumber\\
  & = r_n^{-\alpha}(1+o(1))\int_{\mathcal{B}_{0,r_n/2}\setminus\mathcal{B}_{0,1/\varepsilon_n}} (1-\exp(-|\mathbf{x}|_T^{-\alpha}))d\mathbf{x}.
 % & \leq r_n^{-\alpha}(1+o(1))\int_{\mathcal{B}_{0,n}\setminus\mathcal{B}_{0,1/\varepsilon_n}} (1-\exp(-|\mathbf{x}|_T^{-\alpha}))d\mathbf{x}.
\end{align}
Again, changing to spherical coordinates then yields
\begin{align}\label{eq:boundballclose2}
    &\int_{\mathcal{B}_{0,r_n/2}\setminus\mathcal{B}_{0,1/\varepsilon_n}} (1-\exp(-|\mathbf{x}|_T^{-\alpha}))d\mathbf{x}\nonumber\\
    & \leq \int_{1/\varepsilon_n}^{r_n/2}  \int_{0}^{2\pi}\int_0^\pi\dots \int_0^\pi \rho^{d-1}(1-\exp(-\rho^{-\alpha}))\sin(\phi_1)^{d-2}\dots \sin(\phi_{d-2})d\phi_1 \dots d\phi_{d-1} d\rho \nonumber\\
    & = \frac{2\pi^{d/2}}{\Gamma(\frac{d}{2})} \int_{1/\varepsilon_n}^{r_n/2}  \rho^{d-1}(1-\exp(-\rho^{-\alpha}))d\rho.
   % & \leq \hat{C}r_n^{d-\alpha-1}
\end{align}

Combining~\eqref{eq:boundballfar}~\eqref{eq:boundballfar2},~\eqref{eq:boundballclose1}~\eqref{eq:boundballclose2} gives that
\begin{align}
     & \int_{\mathcal{T}_n^d\setminus({\mathcal{B}_{\mathbf{r_n},1/\varepsilon_n}\cup \mathcal{B}_{\mathbf{0},1/\varepsilon_n}})}(1-\exp(-|\mathbf{x}|_T^{-\alpha}))(1-\exp(-|\mathbf{x}-\mathbf{r}_n|_T^{-\alpha}))dx \nonumber\\
     & \leq  \tilde{C} r_n^{-\alpha}\int_{1/\varepsilon_n}^\infty  \rho^{d-1}(1-\exp(-\rho^{-\alpha}))d\rho\nonumber\\
     & = O(r_n^{-\alpha}\varepsilon_n^{\alpha-d}).
\end{align}
for some $\tilde{C}>0$.

Thus, as $d<\alpha$, this contribution is $o(r_n^{-\alpha})$ because $\varepsilon_n\to 0$. Therefore, 
\begin{equation}\label{eq:contrxlarge}
   \int_{\mathcal{T}_n^d\setminus({\mathcal{B}_{\mathbf{r_n},1/\varepsilon_n}\cup \mathcal{B}_{\mathbf{0},1/\varepsilon_n}})}(1-\exp(-|\mathbf{x}|_T^{-\alpha}))(1-\exp(-|\mathbf{x}-\mathbf{r}_n|_T^{-\alpha}))dx = o(r_n^{-\alpha}).
\end{equation}
Combining this with the integrals over $\mathcal{B}_{\mathbf{r_n},1/\varepsilon_n}$ and $\mathcal{B}_{0,1/\varepsilon_n}$ of~\eqref{eq:tnlb}, results in
\begin{equation*}
    \lim_{n\to\infty}\frac{\Exp{T_n(r_n)}}{ r_n^{-\alpha}}\leq \frac{2\pi^{d/2}\Gamma(1-\frac{d}{\alpha})}{\Gamma(1+\frac{d}{2})}.
\end{equation*}
\end{proof}

We now integrate over all edges of length at least $r_n$ to obtain the expected number of triangles with one edge of length at least $r_n$:

\begin{lemma} \label{lem: wn}
Let $1 \ll r_n\ll n$ and $\alpha>d$. Then,
    \begin{equation*}
    \lim_{n\to\infty}\frac{\Exp{W_n^3(r_n)}}{n^{d}r_n^{d-2\alpha}} = \frac{dC_d\pi^{d}\Gamma(1-\frac{d}{\alpha})}{\Gamma(1+\frac{d}{2})^2(2\alpha-d)}.
\end{equation*}
\end{lemma}
\begin{proof}[Proof of Lemma~\ref{lem: wn}]
Lemma~\ref{lem:Tn} gives the expected number of triangles containing a given edge of length $r_n$. 
To obtain the expected number of triangles from one vertex at the origin that contain at least one edge of length at least $r_n$, we integrate $T_n(x)$ and the probability $g(\mathbf{x},0)$ that a vertex at the origin forms an edge with a vertex at location $\mathbf{x}$ over the region with distance at least $r_n$ from the origin. Thus, using the notation $g(r,0)$ for $g(\mathbf{x},0)$ with $|\mathbf{x}|_T = r$, we get
%To obtain the expected number of triangles from one vertex at the origin that contain at least one edge of length at least $r_n$, as in Lemma~\ref{lem: wn}, integrating over all coordinates with distance at least $r_n$ from the origin gives
\begin{align}\label{eq:intlongedge}
    & \frac{2\pi^{d/2}\Gamma(1-\frac{d}{\alpha})}{\Gamma(1+\frac{d}{2})} \int_{r_n}^\infty\int_{0}^{2\pi}\int_0^\pi\dots \int_0^\pi \rho^{d-1-\alpha}g(\rho,0)\sin(\phi_1)^{d-2}\dots \sin(\phi_{d-2})d\phi_1 \dots d\phi_{d-1} d\rho\nonumber\\
    &  = \frac{4\pi^{d}\Gamma(1-\frac{d}{\alpha})}{\Gamma(\frac{d}{2})\Gamma(1+\frac{d}{2})}\int_{r_n}^\infty \rho^{d-1-2\alpha}d\rho \big( 1 + o(1) \big) \nonumber\\
    & =  \frac{4\pi^{d}\Gamma(1-\frac{d}{\alpha})r_n^{d-2\alpha}}{\Gamma(\frac{d}{2})\Gamma(1+\frac{d}{2})(2\alpha-d)}\big( 1 + o(1) \big).
\end{align}
As the volume of the $d$-dimensional torus equals $C_dn^d$ and the intensity of the Poisson point process equals 1, there are on average $C_dn^d$ vertices, and in this manner every triangle is counted twice. Thus,
    \begin{equation}\label{eq1}
    \lim_{n\to\infty}\frac{\Exp{W^3_n(r_n)}}{n^{d}r_n^{d-2\alpha}} =  \frac{2C_d\pi^{d}\Gamma(1-\frac{d}{\alpha}) }{\Gamma(\frac{d}{2})\Gamma(1+\frac{d}{2})(2\alpha-d)}=  \frac{dC_d\pi^{d}\Gamma(1-\frac{d}{\alpha})}{\Gamma(1+\frac{d}{2})^2(2\alpha-d)},
\end{equation}
where we have used that $x\Gamma(x)=\Gamma(x+1)$.

% \begin{align}\label{eq:intlongedge}
%     & 2^dd!\int_{r_n}^\infty\int_0^{x_{d}}\dots \int_0^{x_{2}}x_d^{-\alpha}2^{d+1}\Gamma(1-d/\alpha)g(\mathbf{x},0)dx_1\dots dx_d \nonumber\\
%     &  = 2^dd\int_{r_n}^\infty x_d^{d-1-2\alpha}2^{d+1}\Gamma(1-d/\alpha) dx_d \big( 1 + o(1) \big) \nonumber\\
%     & = \frac{2^ddr_n^{d-2\alpha}}{2\alpha-d}2^{d+1}\Gamma(1-d/\alpha) \big( 1 + o(1) \big).
% \end{align}
% Now there are on average $C_dn^d$ vertices, and in this manner every triangle is counted twice. Thus,
%     \begin{equation}\label{eq1}
%     \lim_{n\to\infty}\frac{\Exp{W^3_n(r_n)}}{n^{d}r_n^{d-2\alpha}} = \frac{2^{2d}d}{2\alpha-d}C\Gamma(1-d/\alpha).
% \end{equation}
% The expected number of edges of length at least $r_n$ equals
% \begin{equation}\label{eq:nlongedges}
%     N\int_{\mathcal{B}_{0,n}^d: |x|_T>r_n}g(x,\mathbf{0})dx = \frac{r_n^{d-\alpha}N}{2} 2^d\Gamma(1-d/\alpha)(1+o(1))= r_n^{d-\alpha}C_dn^d2^{d-1}\Gamma(1-d/\alpha)(1+o(1)),
% \end{equation}
% where $N$ denotes the number of vertices, and $C_dn^d$ the volume of the torus. 
% Therefore, combining this with Lemma~\ref{lem:Tn} yields
% \begin{equation} \label{eq1}
%     \lim_{n\to\infty}\frac{\Exp{W_n(r_n)}}{n^{d}r_n^{d-2\alpha}} = 2^{2d}C\Gamma(1-d/\alpha)^{2}
% \end{equation}
\end{proof}

\subsection{Generalizing to $k$-cliques}
In this subsection, we generalize Lemma~\ref{lem: wn} from triangles to $k$-cliques to obtain the expected number of $k$-cliques with one edge of length at least $r_n$:
\begin{lemma}\label{lem:cliqueexp}
Let $1\ll r_n\ll n$ and $\alpha>d$. Then, for all $k\geq 3$,
    \begin{equation*}
    \lim_{n\to\infty}\frac{\Exp{W_n^k(r_n)}}{n^{d}r_n^{d-(k-1)\alpha}} =  \frac{dC_d\pi^{d/2} M_k}{2\Gamma(1+\frac{d}{2})((k-1)\alpha-d)},
\end{equation*}
where $M_k$ is as in \eqref{mk}.
\end{lemma}
\begin{proof}
Define $A_k(r_n)$ as the number of $k$-cliques that a given edge of length $r_n$ is part of. We first compute its expected value:
\begin{equation}\label{eq:EKn}
    \Exp{A_k(r_n)} = \int_{\mathcal{T}_n^d}\dots \int_{\mathcal{T}_n^d}\prod_{i=1}^{k-2}g(\mathbf{x_i},\mathbf{0})g(\mathbf{x_i},\mathbf{r_n})\prod_{1\leq u<v\leq k-2}g(\mathbf{x_u},\mathbf{x_v})d\mathbf{x_1}\dots d\mathbf{x_{k-2}}.
\end{equation}
Again, we lower bound by the contribution from  $\mathbf{x}\in \mathcal{B}_{0,1/\varepsilon_n}$, such that $g(\mathbf{x},\mathbf{r_n})= r_n^{-\alpha}(1+o(1))$,
%, the part of the torus where $|x_i^{(j)}|_T\leq 1/\varepsilon_n$ for all $j\in[d]$ and $i\in[k-2]$, 
and $\mathbf{x}\in\mathcal{B}_{\mathbf{r_n},1/\varepsilon_n}$ such that $g(\mathbf{x},\mathbf{0})= r_n^{-\alpha}(1+o(1))$.
%, the part of the torus where $|\mathbf{r}_n-x_i^{(j)}|_T\leq 1/\varepsilon_n$ for all $j\in[d]$ and $i\in[k-2]$. 
%On these areas, $g(\mathbf{x},\mathbf{r_n})= r_n^{-\alpha}(1+o(1))$ or $g(x_i,\mathbf{0})= r_n^{-\alpha}(1+o(1))$. 
As a lower bound, this gives
\begin{align}\label{eq:Knesp}
     \Exp{A_k(r_n)} & \geq  2r_n^{-(k-2)\alpha}\int_{\mathcal{B}_{0,1/\varepsilon_n}}\dots \int_{\mathcal{B}_{0,1/\varepsilon_n}}\prod_{i=1}^{k-2}g(\mathbf{x_i},\mathbf{0})\nonumber\\
     & \quad \quad \times \prod_{1\leq u<v\leq k-2}g(\mathbf{x_u},\mathbf{x_v})d\mathbf{x_1}\dots d\mathbf{x_{k-2}}(1+o(1)).
\end{align}
Taking the limit of $n\to\infty$ then yields
\begin{align}\label{eq:Knklim}
    &\lim_{n\to\infty} \frac{\Exp{A_k(r_n)}}{r_n^{-(k-2)\alpha}} \geq \nonumber\\
    & 2\int_{[-\infty,\infty]^d}\dots \int_{[-\infty,\infty]^d}\prod_{i=1}^{k-2}g(\mathbf{x_i},\mathbf{0})\prod_{1\leq u<v\leq k-2}g(\mathbf{x_u},\mathbf{x_v})d\mathbf{x_1}\dots d\mathbf{x_{k-2}},
\end{align}
where the factor 2 arises from the symmetric contribution from $\mathcal{B}_{\mathbf{r_n},1/\varepsilon_n}$.

As an upper bound, we bound the probability that a $k$-clique with long edge between 0 and $\mathbf{r}_n$ appears by the probability that $k-2$ triangles appear with the vertices at $0$ and at $r_n$. This can be written as
\begin{equation} \label{eq:invoke}
   \Exp{A_k(r_n)} \leq   \Big(\int_{\mathcal{T}_n^d}g(\mathbf{x},\mathbf{0})g(\mathbf{x},\mathbf{r_n})d\mathbf{x}\Big)^{k-2}.
\end{equation}
Now the contribution to this integral from $\mathbf{x}\notin \mathcal{B}_{0,1/\varepsilon_n}\cup \mathcal{B}_{\mathbf{r_n},1/\varepsilon_n}$ can be bounded using ~\eqref{eq:contrxlarge}. Thus, the expected number of $k$-cliques with at least one edge of length $r_n$ and  $\mathbf{x}\notin \mathcal{B}_{0,1/\varepsilon_n}\cup \mathcal{B}_{\mathbf{r_n},1/\varepsilon_n}$,  $\bar{A}_k(r_n,\varepsilon_n)$, satisfies
\begin{align}
    \Exp{\bar{A}_k(r_n,\varepsilon_n)} &  \leq Kr_n^{-\alpha(k-2)}\Bigg( \int_{1/\varepsilon_n}^\infty x^{d-1}(1-\exp(-x^{-\alpha}))dx \Bigg) ^{k-2}\nonumber\\
    & = O(r_n^{-\alpha(k-2)}\varepsilon_n^{d-\alpha})=o(r_n^{-\alpha(k-2)}),
\end{align}
where the last equality holds because $\varepsilon_n\to 0$ as $n\to\infty$. 
Thus, this contribution is sufficiently small compared to the contribution from $\mathcal{B}_{\mathbf{0},1/\varepsilon_n}$ and $\mathcal{B}_{\mathbf{r_n},1/\varepsilon_n}$ of~\eqref{eq:Knesp} when $n$ becomes large. This shows that 
\begin{align}\label{eq:ank}
    & \lim_{n\to\infty} \frac{\Exp{A_k(r_n)}}{r_n^{-(k-2)\alpha}} \nonumber\\
    & =  2\int_{[-\infty,\infty]^d}\dots \int_{[-\infty,\infty]^d}\prod_{i=1}^{k-2}g(\mathbf{x_i},\mathbf{0})\prod_{1\leq u<v\leq k-2}g(\mathbf{x_u},\mathbf{x_v})d\mathbf{x_1}\dots d\mathbf{x_{k-2}}.
\end{align}
We then integrate this to obtain the average number of $k$-cliques with one long edge centered at a single vertex at the origin. Using that the volume of the $d$-dimensional torus is $C_dn^d$, this yields similarly to~\eqref{eq:intlongedge} and~\eqref{eq1} that
    \begin{align}\label{eq:expwnk}
    \lim_{n\to\infty}\frac{\Exp{W_n^k(r_n)}}{n^d r_n^{d-(k-1)\alpha}}&  = \frac{\pi^{d/2} M_kC_d}{\Gamma(\frac{d}{2})((k-1)\alpha-d)}\nonumber\\
    & = \frac{dC_d\pi^{d/2} M_k}{2\Gamma(1+\frac{d}{2})((k-1)\alpha-d)},
\end{align}
where we have used that $d/2\Gamma(d/2)=\Gamma(1+d/2)$. 
Similarly, we obtain that 
\begin{equation}\label{eq:barWnk}
    \Exp{\bar{W}_k(r_n,\varepsilon_n)} = O(\varepsilon_n^{d-\alpha}n^dr_n^{d-(k-1)\alpha}),
\end{equation}
where $\bar{W}_n^k(r_n,\varepsilon)$ denotes the number of $k$-cliques with at least one edge of length at least $r_n$ that are not in $W_n^k(r_n,\varepsilon)$.
\end{proof}

Finally, we bound the variance of the number of cliques with $k-1$ long edges and all other edges of length at most $1/\varepsilon$: \
\begin{lemma}\label{lem:varcliques}
Suppose that $\varepsilon_n\geq 1/\log(n)$ and that $\lim_{n\to \infty}\varepsilon_n=0$. When $r_n\ll n^{d/((k-1)\alpha-d)}$,
    \begin{equation*}
        \lim_{n\to\infty}\frac{\Var{W_n^k(r_n,\varepsilon_n)}}{\Exp{W_n^k(r_n,\varepsilon_n)}^2} = 0.
    \end{equation*}
\end{lemma}
\begin{proof}
First of all, we can interpret $W_n^k(r_n,\varepsilon)$ as a $U$-statistic of order $k$ on a marked space. Because of this,~\cite[Lemma 3.5]{Reitzner2013} yields that
\begin{align}\label{eq:varustat}
   &  \Var{W_n^k(r_n,\varepsilon)}= \sum_{s=1}^ks!{k \choose s}^2\int_{\mathcal{T}_n^d}\dots \int_{\mathcal{T}_n^d} \nonumber\\
    & \times\Bigg(\int_{\mathcal{B}^{d}_{0,n}}\dots \int_{\mathcal{B}^{d}_{0,n}} f(\mathbf{x}_1,\dots,\mathbf{x}_s,\mathbf{y}_1,\dots,\mathbf{y}_{k-s}) d\mathbf{y}_1\dots  d\mathbf{y}_{k-s}\Bigg)^2 d\mathbf{x}_1\dots  d\mathbf{x}_s
\end{align}
where $f(\mathbf{x}_1,\dots,\mathbf{x}_s,\mathbf{y}_1,\dots,\mathbf{y}_{k-s})$ denotes the probability that $\mathbf{x}_1,\dots,\mathbf{x}_s,\mathbf{y}_1,\dots,\mathbf{y}_{k-s}$ forms a clique with $k-1$ edges of length at least $r_n$, and that all other edges have length at most $\varepsilon_n$. Thus, each term inside the summation measures the expected number of two 'glued' together cliques with overlap at $s$ vertices.

Define
\begin{equation}
    \tilde{g}(x)=1-\exp(-x^{-\alpha}).
\end{equation}

Now cliques in $W_n^k(r_n,\varepsilon_n)$ have $k-1$ long edges of length at least $r_n$, while all other edges have length at most $1/\varepsilon_n$.
    We first focus on the contribution to the summand in~\eqref{eq:varustat} from $s\in{ 2,\dots, k-1}$. When the overlap between the cliques occurs at the vertex incident with the long edges of length $\geq r_n$, then $s-1$ long edges overlap (and several short edges possibly as well). In this case, using~\eqref{eq:ank} shows that this summand equals
    \begin{align}\label{eq:varcontrlongedge}
       & \bigO{ \Exp{W_n^k(r_n,\varepsilon_n)}\Exp{A_{k-(s-1)}(r_n)} \tilde{g}(\varepsilon_n^{-1})^t}\nonumber\\
       &= \bigO{ \Exp{W_n^k(r_n)}\Exp{A_{k-(s-1)}(r_n)} \tilde{g}(\varepsilon_n^{-1})^t}\nonumber\\
       & = \bigO{n^dr_n^{d-2(k-1)\alpha+s\alpha }  \tilde{g}(\varepsilon_n^{-1})^t} \nonumber\\
       & =\bigO{n^dr_n^{d-(k-1)\alpha}  \tilde{g}(\varepsilon_n^{-1})^t} ,
        %= \bigO{\frac{\Exp{W_n^k(r_n,\varepsilon)\mid N}^2}{r_n^{-\alpha(s-1)}Nr_n^{d-\alpha}}}.
    \end{align}
    for some $t\geq 0$.
    Here $A_{k-(s-1)}(r_n)$ is defined as in the proof of Lemma \ref{lem:cliqueexp}. Indeed, to glue two cliques together at $s-1$ long edges, one first needs to create one clique with $k-1$ long edges, which gives the $\Exp{W_n^k(r_n,\varepsilon)}$ term of Lemma~\ref{lem:cliqueexp}. Given the presence of this clique, when sharing $s-1$ long edges, this means that to generate the second clique, only a clique with $k-s$  extra long edges has to be added to generate the second clique (plus some additional short edges which appear with $O(\tilde{g}(1/\varepsilon_n))$ probability). This is equivalent to generating a new clique of size $k-s+1$, (where the +1 comes from the fact that the vertex incident to the long edges is also counted)  %\cs{Here the computations are again correct, as a clique with $k-s$ long edges, has size $k-s+1$} 
    from a given edge of length at least $r_n$, giving a contribution of $\Exp{A_{k-(s-1)}(r_n)}$, computed in~\eqref{eq:ank}. 

    When $s=k$, then the term in the summand of~\eqref{eq:varustat} equals $\Exp{W_n^k(r_n,\varepsilon)}$.
    
    When the overlap occurs at only short edges however or when $s=1$ and no edges overlap at all, the second clique still needs to contain $k-1$ long edges. As the probability of a short edge is $O(g(1/\varepsilon_n))$, this means that to generate the second clique there are on average $\Exp{W_n^k(r_n,\varepsilon)}/O(g(1/\varepsilon_n)^{s(s-1)/2}n^{ds})$ options, as there are $s(s-1)/2$ short edges in the overlapping part of the cliques and on average $O(n^{ds})$ sets of $s$ vertices. combined, this gives a contribution to the summand in~\eqref{eq:varustat} of
    \begin{align}\label{eq:varcontrshortedge}
        & \bigO{ \Exp{W_n^k(r_n,\varepsilon_n)}^2/\tilde{g}(\varepsilon_n^{-1})^{s(s-1)/2}n^{ds}} \nonumber\\
        &= \bigO{n^{2d-sd}r_n^{2(d-(k-1)\alpha)}/\tilde{g}(\varepsilon_n^{-1})^{-s(s-1)/2}} \nonumber\\
        & = \bigO{n^{d}r_n^{2(d-(k-1)\alpha)}/\tilde{g}(\varepsilon_n^{-1})^{-k(k-1)/2}}
    \end{align}
    Now $\tilde{g}(\varepsilon_n^{-1}) = O(\varepsilon_n^\alpha)$. As by assumption, $\varepsilon_n\geq \log(n)^{-1}$ and $k$ is fixed,~\eqref{eq:varustat},~\eqref{eq:varcontrlongedge} and~\eqref{eq:varcontrshortedge} combined yield that 
\begin{equation}\label{eq:varbound}
    \Var{W_n^k(r_n,\varepsilon_n)}  = \bigO{n^dr_n^{d-(k-1)\alpha}\tilde{g}(\varepsilon_n^{-1})^{-t}}
\end{equation}
for some $t>0$. 

Now by~\eqref{eq:expwnk} and~\eqref{eq:barWnk} and the fact that $\varepsilon_n\to 0$ as $n\to\infty$, $\Exp{W_n^k(r_n,\varepsilon)}^2 = O(n^{2d}r_n^{2(d-(k-1)\alpha})$. Combining this with~\eqref{eq:varbound} yields that for fixed $\varepsilon$,
    \begin{equation}
        \frac{\Var{W_n^k(r_n,\varepsilon_n)}}{\Exp{W_n^k(r_n,\varepsilon_n)}^2} = O(n^{-d}r_n^{(k-1)\alpha-d}\varepsilon^{-\alpha t}).
    \end{equation}
    Now this expression tends to zero when $r_n\ll n^{d/((k-1)\alpha-d)}$ and $\varepsilon_n\geq \log(n)^{-1}$.

\end{proof}

%\subsection{Proving Theorem~\ref{thm:localization_long_cliques}}
We are now ready to prove Theorem~\ref{thm:localization_long_cliques}:

\begin{proof}[Proof of Theorem~\ref{thm:localization_long_cliques}]
    First of all,
    \begin{equation*}
        W_n^k(r_n) = W_n^k(r_n,\varepsilon_n) + \bar{W}_n^k(r_n,\varepsilon_n),
    \end{equation*}
    where $\bar{W}_n^k(r_n,\varepsilon_n)$ denotes the number of $k$-cliques with at least one edge of length at least $r_n$ that are not in $W_n^k(r_n,\varepsilon_n)$. Now by~\eqref{eq:barWnk} %\textcolor{red}{do you mean another equation?},
    \begin{equation*}
        \Exp{\bar{W}_n^k(r_n,\varepsilon_n)} = \bigO{n^dr_n^{d-(k-1)\alpha} \varepsilon_n^{d-\alpha}}.
    \end{equation*}
    Thus, by the Markov inequality,
    \begin{equation*}
        W_n^k(r_n) = W_n^k(r_n,\varepsilon_n) + \bigOp{n^dr_n^{d-(k-1)\alpha} \varepsilon_n^{d-\alpha}}.
    \end{equation*}
    Now by Lemma~\ref{lem:varcliques} and the Chebyshev inequality
    \begin{equation*}
        W_n^k(r_n) = \Exp{W_n^k(r_n,\varepsilon_n)}(1+\op(1)) + \bigOp{n^dr_n^{d-(k-1)\alpha}\varepsilon_n^{d-\alpha}}.
    \end{equation*}
    Now using that as in~\eqref{eq:expwnk}
    \begin{equation*}
        \lim_{n\to \infty}\frac{\Exp{W_n^k(r_n,\varepsilon_n)}}{n^dr_n^{d-(k-1)\alpha}} =  \frac{dC_d\pi^{d/2} M_k}{2\Gamma(1+\frac{d}{2})((k-1)\alpha-d)}
    \end{equation*}
    %\textcolor{red}{where did we compute this limit? is it the correct one?}
    finishes the first part of the proof of the theorem. It remains to prove the convergence in distribution. By Lemma \ref{lem:cliqueexp} for sufficiently large $n \in \mathbb{N}$ we have
\begin{equation} \label{1rev2}
    \Exp{W_n^k(r_n)}  \geq cn^d r_n^{d-(k-1)\alpha}
\end{equation}    
    for some constant $c \in (0, \infty)$. Moreover, in particular we observe that
    \begin{equation}\label{eq:limexpinf}
        \lim_{n \to \infty} \Exp{W_n^k(r_n)} = \infty
    \end{equation} 
    by the assumption that $1\ll r_n\ll n^{d/((k-1)\alpha-d)}$. By Lemma \ref{TVP} below, for a random variable $P_n$ having a Poisson distribution with parameter $\Exp{W_n^k(r_n)}$:
    \begin{align*}
        d_{\text{TV}}\left(W^k_n(r_n),P_n(\beta)\right)&\leq c \Exp{W_n^k(r_n)}^{-1} n^d r_n^{2d-2(k-1)\alpha} \\
        & \leq c r_n^{d-(k-1)\alpha},
    \end{align*}
    which tends to zero as $n \to \infty$, where we have used \eqref{1rev2} in the last inequality. Moreover, since $P_n$ has a Poisson distribution with parameter $\Exp{W_n^k(r_n)}$ satisfying ~\eqref{eq:limexpinf}, we have that 
    \begin{equation*}
    \frac{P_n - \mathbb{E}[W^k_n(r_n)] }{\sqrt{\mathbb{E}[W^k_n(r_n)]}}  \xrightarrow[n\to \infty]{d} Z,
\end{equation*}
where $Z$ is a random variable with standard normal distribution. Thus, the same conclusion holds for $W^k_n(r_n)$, since $d_{\text{TV}}\left(W^k_n(r_n),P_n(\beta)\right)$ converges to 0, as $n\to \infty$. This finishes the proof.
\end{proof}

\section{Proof of Theorem~\ref{lem: wnc}}  \label{pr3}

By Lemma \ref{lem:cliqueexp}, when choosing $$r_n  = \Bigg( \frac{2\Gamma(1+\frac{d}{2})((k-1)\alpha-d)}{dC_d\pi^{d/2} M_k}\Bigg)^{\frac{1}{d-(k-1)\alpha}} r n^{\frac{d}{(k-1)\alpha-d}},$$ with $r \in (0,\infty)$, then
$$\lim_{n\to\infty}{\Exp{W^k_n(r_n)}} = r^{d-(k-1)\alpha}. $$
This will be used later to conclude the suitable Poisson limit, from which we will subsequently derive the extreme value behavior of the longest edge belonging to a $k$-clique. We first recall a formal way of constructing the underlying geometric graph by means of a marked Point process. This will later be useful to apply a Poisson approximation theorem.

\subsection{Formal construction of the random graph model} \label{formalc}

Let $( \mathbb{M}, \mathcal{M}, m)$ be a mark space, where $\mathbb{M}=[0,1)^{\mathbb{N}_0}$, $\mathcal{M}$ is a corresponding $\sigma$-algebra and $m$ is the distribution of an infinite sequence of independent $[0,1)$-uniformly distributed random variables. We let $\eta$ be a marked Poisson point process of unit intensity on $\mathcal{T}_n^d$ with marks in $\mathbb{M}$ distributed according to $m$. Thus, $\eta$ is a Poisson point process on $\mathcal{T}_n^d \times \mathbb{M}$. We can denote almost every realization of $\eta$ by
$$\eta = \Big\{ (x_i,(t_{i,0}, t_{i,1}, \dots)): {i=1,\dots, |\eta|} \Big\} ,$$
where $t_{1,0}<t_{2,0}\dots <t_{|\eta|,0}$, i.e. we use the first coordinates in order to determine the order in which the points in $\eta$ are enumerated. Then there is an edge $\{x_i,x_j\}$, $1\leq i <j\leq |\eta|$, if and only if $t_{i,j}\leq  g(x_i,x_j)$. For a set $\eta$ of marked points as above, we denote by $G(\eta)=(V(\eta),E(\eta))$ the random geometric graph constructed from $\eta$. 

\subsection{Poisson approximation}

We denote by $\mathbf{S}$ the set of all locally finite subsets of $\mathcal{T}_n^d\times\mathbb{M}$ and by $\mathbf{S}_l$ the set of subsets of $\mathcal{T}_n^d\times\mathbb{M}$ with cardinality $l$, $l\in\mathbb{N}$. We recall the following result.

\begin{lemma}[{\cite[Theorem 3.1.]{Penrose}}] \label{Th:Penrose18}

Let $l\in\mathbb{N}$, $f\colon \mathbf{S}_l\times \mathbf{S} \longrightarrow \{0,1\}$ a measurable function and for $\xi\in \mathbf{S}$ set
\begin{equation} \label{coupf}
	F(\xi):=\sum_{\psi\in\mathbf{S}_l:\psi\subset\xi}f(\psi,\xi\setminus \psi).
\end{equation}

Let $\eta$ be a marked Poisson point process, set $W:=F(\eta)$ and $\beta:=\mathbb{E}[W]$. For $x_1,\dots,x_l\in \mathcal{T}_n^d$ set 
$$p(x_1,\dots,x_l):=\mathbb{E}\left[f\left(\{(x_1,\tau_1),\dots,(x_l,\tau_l)\},\eta \right)\right],$$
where the $\tau_i$ are independent random variables in $\mathbb{M}$ with common distribution $\mathbf{m}$. 

Suppose that for almost every $\mathbf{x}=(x_1,\dots,x_l)$, $x_i\in \mathcal{T}_n^d$, with $p(x_1,\dots,x_l)>0$ we can find coupled random variables $U_\mathbf{x}$ and $V_\mathbf{x}$ such that the following holds:
\begin{enumerate}[label={\upshape(\roman*)}]
\item $W\overset{d}{=}U_\mathbf{x}$,
\item $F\left(\eta \cup\overset{l}{\underset{i=1}{\bigcup}}\{(x_i,\tau_i)\}\right)$ conditional on $f\left(\overset{l}{\underset{i=1}{\bigcup}}\{(x_i,\tau_i)\},\eta\right)=1$ has the same distribution as $1+V_\mathbf{x}$,
\item $\mathbb{E}\left[\vert U_\mathbf{x}-V_\mathbf{x}\vert\right]\leq w(\mathbf{x})$, for some measurable function $w$.
\end{enumerate}

Let $P(\beta)$ be a random variable with Poisson distribution and mean $\beta$. Then
\begin{equation}d_{\text{TV}}\left(W,P(\beta)\right)\leq\frac{\min(1,\beta^{-1})}{l!}\int_{\mathcal{T}_n^d \times \ldots \times \mathcal{T}_n^d}w(\mathbf{x})p(\mathbf{x})\operatorname{d} \mathbf{x},\label{th:Pdtv}\end{equation}
where $d_{TV}$ denotes the total variation distance between two discrete random variables. \newline
\end{lemma}

Formally, write $\mathcal{C}_k$ for the set of all $k$-cliques, i.e. the set of all vertices $x_1, \ldots, x_k$ that form a clique. Making use of the formal construction described above, we will apply Lemma \ref{Th:Penrose18} with $l=k$ and $f \colon \mathbf{S}_k\times \mathbf{S}$ defined by
{\footnotesize
\begin{align*}
     & f\left( \left\{\left( x_1,(u^1_i)_{i\geq 0}\right) , \ldots \left( x_k,(u^k_i)_{i\geq 0}\right)\right\},\xi \right) \\
    & =\mathbf{1}\{ C_k:=\{x_1, \ldots, x_k\} \in \mathcal{C}_k, \max_{x,y \in C_k}|x-y|_T  \mathbf{1}_{\{\{x,y\} \in E(\xi\cup \{( x_1,(u^1_i)_{i\geq 0}) , \ldots ( x_k,(u^k_i)_{i\geq 0})\})\}}>r_n\} 
\end{align*}}
\noindent for $\left\{\left( x_1,(u^1_i)_{i\geq 0}\right) , \ldots \left( x_k,(u^k_i)_{i\geq 0}\right)\right\} \in \mathbf{S}_k$, so that $ W=F(\eta)=W^k_n(r_n)$. By definition, $W$ is the number of $k$ vertices that form a $k$-clique and contain an edge longer than $r_n$. One can formally check that $f$ is the indicator function of a measurable set.

Now we define the coupled random variables $U_\mathbf{x}$ and $V_\mathbf{x}$. To this end, we enlarge $\eta$ by adding $k$ marked points $x_1^{\text{m}}=(x_1, (u^1_i)_{i\geq 0}), \ldots, x_k^{\text{m}}=(x_k, (u^k_i)_{i\geq 0})$. Then we define $V_\mathbf{x}$ as the number of $k$ vertices other than $\{x_1, \ldots, x_k\}$ in the enlarged graph $G(\eta \cup \{x_1^{\text{m}}, \ldots, x_k^{\text{m}}\} )$ that both form a clique and contain an edge longer than $r_n$:
\begin{align*}
 & V_\mathbf{x}= \sum\nolimits_{1}
  \mathbf{1}\{\max|x-y|_T  \mathbf{1}_{\{\{x,y\} \in E(\xi\cup \{\left( x_1,(u^1_i)_{i\geq 0}\right) , \ldots \left( x_k,(u^k_i)_{i\geq 0}\right)\})\}}>r_n\} ,
\end{align*}
%_{y_1, \ldots, y_k \in V (G(\eta \cup \{x_1^{\text{m}}, \ldots, x_k^{\text{m}}\} )): \{y_1, \ldots, y_k\} \neq \{x_1, \ldots, x_k\}}
%_{z \in V(\eta \cup \{x^{\text{m}}\} )}
where the sum $\sum_1$ is taken over all $\{y_1, \ldots, y_k\} \neq \{x_1, \ldots, x_k\}$ with $y_1, \ldots, y_k \in V (G(\eta \cup \{x_1^{\text{m}}, \ldots, x_k^{\text{m}}\} ))$ that form a clique and the maximum in the indicator is with respect to all endpoints $x,y \in \{y_1, \ldots, y_k\}$.

Now consider the restriction of $G(\eta \cup \{x_1^{\text{m}}, \ldots, x_k^{\text{m}}\})$ to $\eta$, that is the subgraph of $G(\eta \cup \{x_1^{\text{m}}, \ldots, x_k^{\text{m}}\})$ by deleting the vertices $x_1, \ldots, x_k$ and all edges having one of them as an endpoint. We denote this graph by $\eta_R=G(\eta \cup \{x_1^{\text{m}}, \ldots, x_k^{\text{m}}\})_{|\eta} =(V(\eta_r),E(\eta_R))$ and we note that it has the same distribution as $G(\eta)$. Remark that these two graphs are not equal almost surely, because the presence of an edge not only depends on the marks of its endpoints, but also on all other marks, as they are used in order to determine the order according to which the points are enumerated. The random variable $U_\mathbf{x}$ is then defined as the number of $k$ vertices in the induced graph $\eta_R$ that both form to a $k$-clique and contain an edge longer than $r_n$:
$$ U_\mathbf{x}= \sum\nolimits_2 \mathbf{1}\{\max |z-y|_T \mathbf{1}_{\{ \{y, z\} \in E(\eta_R) \}} > r_n\},$$
where now the sum $\sum_2$ is taken over all cliques $\{y_1, \ldots, y_k\}$ with $y_1, \ldots, y_k \in V (\eta_r) $ and the maximum in the indicator is with respect to all endpoints $z,y \in \{y_1, \ldots, y_k\}$. Observe that the coupled random variables $U_\mathbf{x}$ and $V_\mathbf{x}$ defined above satisfy Assumptions (i) and (ii)  of Theorem \ref{Th:Penrose18} and  are such that $V_\mathbf{x}\geq U_\mathbf{x}$ by construction. It follows that 
\begin{align*}
 \mathbb{E}\left[\vert U_\mathbf{x}-V_\mathbf{x}\vert\right]&=\mathbb{E}\left[ V_\mathbf{x}-U_\mathbf{x}\right]\\
&=\mathbb{E}\Bigg[ \sum\nolimits_{1}
  \mathbf{1}\{\max|x-y|_T  \mathbf{1}_{\{\{x,y\} \in E(\xi\cup \{\left( x_1,(u^1_i)_{i\geq 0}\right) , \ldots \left( x_k,(u^k_i)_{i\geq 0}\right)\})\}}>r_n\} \Bigg]\\
& \quad  -\mathbb{E}\Bigg[\sum\nolimits_{2}  \mathbf{1}\{\max |z-y|_T \mathbf{1}_{\{ \{y, z\} \in E(\eta_R) \}} > r_n\}\Bigg]\\
&=:w(\mathbf{x}).
%&\leq \mathbb{E}\left[ \sum_{y \in B_k (G(\eta \cup \{x^{\text{m}}\} )): y\neq x} \mathbf{1}_{\{|x-y|_T \mathbf{1}_{\{\{y, z \} \in E(\eta \cup \{x^{\text{m}}\})\}} > r_n\}}\right]=:w(x) .
\end{align*}
Thus, Assumption (iii) of Theorem \ref{Th:Penrose18} is also satisfied. Moreover, from the calculations in the proof of Lemma \ref{lem:cliqueexp} we obtain for $x_1, \ldots, x_k \in \mathcal{T}_n^d$ that
\begin{align*}
w(\mathbf{x}) \leq c r_n^{d-(k-1)\alpha}.
\end{align*}
Similarly, we also have
\begin{align*}
p(\mathbf{x})&=\mathbb{E} \left[  f \left( \{\left( x_1,(u^1_i)_{i\geq 0}\right) , \ldots \left( x_k,(u^k_i)_{i\geq 0}\right)\},\xi \right) \right] \leq c r_n^{d-(k-1)\alpha}.
\end{align*}
Thus, we get the following upper bound for the integral in the r.h.s of \eqref{th:Pdtv}:
\begin{align*}
\int_{\mathcal{T}_n^d}w(\mathbf{x})&p(\mathbf{x})d\mathbf{x} \leq c n^d r_n^{2d-2(k-1)\alpha},
\end{align*}
which proves the following:

\begin{lemma} \label{TVP}
    Let $P_n(\beta)$ be a random variable with Poisson distribution and mean $\beta = \mathbb{E}[W^k_n(r_n)]$, where $1\ll r_n \ll n$. Then
\begin{equation*}
d_{\text{TV}}\left(W^k_n(r_n),P_n(\beta)\right)\leq c \frac{\min(1,\beta^{-1})}{k!} n^d r_n^{2d-2(k-1)\alpha}
\end{equation*}
where $c \in (0, \infty)$ is a constant independent of $n$. \newline
\end{lemma}

Now, we are in position to close the proof of Theorem  \ref{lem: wnc}. Recall again that
$$\lim_{n\to\infty}{\Exp{W^k_n(r_n)}} = r^{d-(k-1)\alpha} $$
in the present case. Moreover, by definition of $r_n$ the expression $n^d r_n^{2d-2(k-1)\alpha}$ tends to 0, as $n \to \infty$. Thus, since the total variation distance between two Poisson random variables is bounded by the absolute
value of the difference of their parameters, Theorem ~\ref{lem: wnc} follows from Lemma \ref{TVP}. Then Corollary \ref{frechet} is an easy consequence.

\begin{remark} \label{ustat2}
    As mentioned in Remark \ref{ustat} one can also use for example \cite[Theorem 7.1]{decreusefond2016} regarding Poisson approximation of U-statistics. However, our bound in Lemma \ref{TVP} involves the additional factor $\min \{1, \beta^{-1}\}$, which is missing in \cite[Theorem 7.1]{decreusefond2016}. This factor thus improves the upper bound in case the expectation is large. Indeed, this is the case in the setting of Theorem \ref{thm:localization_long_cliques}, as the expectation diverges, thus making the factor crucial to prove the central limit theorem.
\end{remark}

\section{Proof of Theorem~\ref{thm:cliques_total}} \label{pr2}

%\begin{proof}[Proof of Theorem~\ref{thm:cliques_total}]
By Lemma~\ref{lem:cliqueexp}, the expected number of cliques with one vertex that is at least $1/\varepsilon_n$ away from the other vertices is upper bounded by $Cn^d\varepsilon_n^{\alpha(k-1)-d}$ for large enough $n$.
%while as a lower bound, the expected number of such cliques is at least \textcolor{red}{is the next inequality needed in the proof?}
Thus, by the Markov inequality,
    \begin{equation}
        K_k = K_k(\varepsilon_n) + \bigOp{n^d\varepsilon_n^{\alpha(k-1)-d}}.
    \end{equation}
    Furthermore, as all edges in $K_k(\varepsilon_n)$ are short edges, by the same arguments as in the proof of Lemma \ref{lem:varcliques}, we obtain
         \begin{equation*}
        \lim_{n\to\infty}\frac{\Var{K_k(\varepsilon_n)}}{\Exp{K_k(\varepsilon_n)}^2} = 0.
    \end{equation*}
    Then, the Chebyshev inequality yields that 
    \begin{equation}
        K_k(\varepsilon_n) = \Exp{K_k(\varepsilon_n)}(1+\op(1)) .
    \end{equation}
        Now, 
    \begin{align*}
     &\Exp{K_k(\varepsilon_n)}  \nonumber\\
     & \geq C_dn^d\int_{\mathcal{B}_{0,1/\varepsilon_n}}\dots \int_{\mathcal{B}_{0,1/\varepsilon_n}}\prod_{i\leq k-1}g(\mathbf{x_i},0)\prod_{1\leq u<v\leq k}g(\mathbf{x_u},\mathbf{x_v})d\mathbf{x_1}\dots d\mathbf{x_{k-1}},
\end{align*}
resulting in
\begin{equation}
        \frac{K_k}{K_k(\varepsilon_n)} = 1 + \frac{\bigOp{n^d\varepsilon_n^{\alpha(k-1)-d}}}{n^d(1+\op(1))},
    \end{equation}
    so that
\begin{equation}
    \frac{K_k}{K_k(\varepsilon_n)}\plim 1.
\end{equation}
%     \textcolor{red}{could we add some details on how to conclude this part?}
\qed %    \end{proof}

\section{Conclusion and discussion}\label{sec:conclusion}
In this paper, we have investigated the distances within cliques of arbitrary size $k$, with a focus on large-deviations behavior of the biggest distance in a $k$-clique for the soft random geometric graph on a $d$-dimensional torus of radius $n$. While all typical cliques contain only low-weighted edges, the largest distance in the $k$-clique scales as $n^{d/((k-1)\alpha-d)}$, indicating that the longest edge in a $k$-clique decreases in $k$, in contrast to the length of a typical edge in any $k$-clique. Furthermore, with high probability, such a clique contains $k-1$ edges of this length, and all other edges are short. We further showed that the properly re-scaled length of the longest edge in a clique converges in distribution to a random variable with Fr\'echet distribution. 

We believe that this leads to several interesting questions for further research. First, it would be interesting to investigate whether this convergence holds for all other possible subgraph counts. We believe that for other subgraph counts, the scaling of the longest edge will only depend on the minimal degree of the subgraph, and not on the subgraph size, and it would be interesting to show this in further research. Another option would be to study asymptotic properties of clique counts and distributional limits of the maximal clique (size) in the soft random geometric graph with a similar approach and methodology as presented in this paper. 

Secondly, it would be interesting to extend these results to different types of geometric graphs, particularly for the discrete version of the model in this paper, see \cite{coppersmith2002diameter}. Here, it has been shown in \cite{rs1} that there are subtle differences for the longest global edge, compared to the continuous version. Other options include investigating other models with different underlying spaces such as hyperbolic random graphs~\cite{krioukov2010}. In these models, the clique-behavior would also be interesting to investigate. 

Moreover, this work shows the likelihood of unlikely cliques with long edges, therefore focusing on the structures of cliques rather than their total counts. It would be interesting to see whether these insights on the behavior of the largest edge-lengths can also lead to hypothesis tests to distinguish different random graph models or to detect anomalies in graphs with statistics such as in~\cite{litvak2022}.

Furthermore, in this paper we work with the particular connection probability~\eqref{cf}. Still, our results can also be transferred to the setting in which this connection probability has exponential decay, where $g(x,y)=\exp(-\lambda|x-y|_T^{-\alpha})$, $\lambda, \alpha>0$. In such a case, we believe that one would obtain a convergence of Gumbel type for the maximal clique distance, similar as for the maximal edge-length in~\cite{rs1}.

Finally, often one is interested in the sum of (powers of) the edge-lengths in geometric graphs (see for example \cite{rst}). As an additional insight into the structure of cliques, it would be interesting to study the total sum of edge-lengths in $k$-cliques, and other edge and/or edge-length functionals.

\paragraph{Index of notation.} 
\begin{align*}
    & \mathcal{T}_n^d: && \textnormal{ }d\textnormal{-dimensional torus of radius } n  \\
    & C_d: && \textnormal{ constant such that the volume of the }d\textnormal{-dimensional torus  }  \\
    & && \textnormal{ has }  \operatorname{vol} (\mathcal{T}_n^d )= C_dn^d\\
    & W_n^k(r_n): && \textnormal{ number of } k\textnormal{-cliques with at least one edge of length at least } r_n \\
    & W_n^k(r_n,\varepsilon): &&\textnormal{ number of } k\textnormal{-cliques with exactly }k-1 \textnormal{ edges of length at least } r_n \\
    & &&\textnormal{ and all other edges of length at most } 1/\varepsilon \\
%    & \plim: & \textnormal{ convergence in probability }  & \\
%    & \dlim: & \textnormal{ convergence in distribution }  & \\
%    & e_n^{k,*}: & \textnormal{ largest distance within a clique of size } k  & \\
    & K_k(\varepsilon_n): && \textnormal{ number of } k\textnormal{-cliques with all interdistances of length at most } \varepsilon_n  \\
   & K_k: && \textnormal{ number of } k\textnormal{-cliques}   \\
    & T_n(r_n): && \textnormal{ number of triangles that a given edge of length } r_n\textnormal{ is part of}   \\
%    & [d]: & \{ 1, \ldots, d\}  & \\
    & \mathcal{B}_{y,1/\varepsilon_n}: && \textnormal{ ball of radius }  1/\varepsilon_n \textnormal{ around }  y  \\
    & A_k(r_n): && \textnormal{ number of } k\textnormal{-cliques that a given edge of length }  r_n \textnormal{ is part of}
%    & f(n) = o(g(n)): & \lim_{n \to \infty} f(n)/g(n)=0  & \\
%    & f(n) = O(g(n)): & \limsup_{n \to \infty} f(n)/g(n) < \infty  & \\
%    & X_n = O_\mathbb{P}(a(n)): & \textnormal{ for every } \varepsilon>0 \textnormal{ there exists } M>0, N\in \mathbb{N} \textnormal{ with } \mathbb{P} ( |X_n|/a(n) \geq M) < \varepsilon, \quad \forall n>N  & \\
%    & ( \mathbb{M}, \mathcal{M}, m): & \textnormal{ mark space }  & \\
%    & \mathbf{S}: & \textnormal{ set of all locally finite subsets of } \mathcal{B}_{0,n}^d \times \mathbb{M} & \\
%    & \mathbf{S}_l: & \textnormal{ set of all subsets of } \mathcal{B}_{0,n}^d \times \mathbb{M} \textnormal{ of cardinality } l&
\end{align*}

\paragraph{Acknowledgements.} C.S. was funded by NWO VENI grant 202.001 and NWO M2 grant 0.397.

\bibliography{lit}

\begin{thebibliography}{25}
\providecommand{\natexlab}[1]{#1}
\providecommand{\url}[1]{{#1}}
\providecommand{\urlprefix}{URL }
\providecommand{\doi}[1]{\url{https://doi.org/#1}}
\providecommand{\eprint}[2][]{\url{#2}}
 \bibcommenthead

\bibitem[{Bianconi and Marsili(2006)}]{bianconi2006number}
Bianconi G, Marsili M (2006) Number of cliques in random scale-free network
  ensembles. Physica D: Nonlinear Phenomena 224(1-2):1--6.
  \doi{10.1016/j.physd.2006.09.013}

\bibitem[{Bl{\"a}sius et~al(2018)Bl{\"a}sius, Friedrich, and
  Krohmer}]{Blasius2017}
Bl{\"a}sius T, Friedrich T, Krohmer A (2018) Cliques in hyperbolic random
  graphs. Algorithmica 80(8):2324--2344. \doi{10.1007/s00453-017-0323-3}

\bibitem[{Bogu{\~{n}}{\'{a}} et~al(2020)Bogu{\~{n}}{\'{a}}, Krioukov, Almagro,
  and Serrano}]{boguna2020}
Bogu{\~{n}}{\'{a}} M, Krioukov D, Almagro P, et~al (2020) Small worlds and
  clustering in spatial networks. Physical Review Research 2(2):023040.
  \doi{10.1103/physrevresearch.2.023040}

\bibitem[{Bringmann et~al(2019)Bringmann, Keusch, and Lengler}]{bringmann2015}
Bringmann K, Keusch R, Lengler J (2019) Geometric inhomogeneous random graphs.
  Theoretical Computer Science 760:35--54. \doi{10.1016/j.tcs.2018.08.014}

\bibitem[{Callaway et~al(2000)Callaway, Newman, Strogatz, and Watts}]{callaway}
Callaway DS, Newman ME, Strogatz SH, et~al (2000) Network robustness and
  fragility: Percolation on random graphs. Physical review letters 85(25):5468

\bibitem[{Cohen et~al(2000)Cohen, Erez, Ben-Avraham, and Havlin}]{cohen}
Cohen R, Erez K, Ben-Avraham D, et~al (2000) Resilience of the internet to
  random breakdowns. Physical review letters 85(21):4626

\bibitem[{Coppersmith et~al(2002)Coppersmith, Gamarnik, and
  Sviridenko}]{coppersmith2002diameter}
Coppersmith D, Gamarnik D, Sviridenko M (2002) The diameter of a long-range
  percolation graph. Random Structures and Algorithms 21(1):1--13.
  \doi{10.1002/rsa.10042}

\bibitem[{Decreusefond et~al(2016)Decreusefond, Schulte, and
  Thäle}]{decreusefond2016}
Decreusefond L, Schulte M, Thäle C (2016) {Functional Poisson approximation in
  Kantorovich–Rubinstein distance with applications to U-statistics and
  stochastic geometry}. The Annals of Probability 44(3):2147--2197.
  \doi{10.1214/15-aop1020}

\bibitem[{Devroye et~al(2011)Devroye, György, Lugosi, and Udina}]{devroye2011}
Devroye L, György A, Lugosi G, et~al (2011) High-dimensional random geometric
  graphs and their clique number. Electronic Journal of Probability
  16:2481--2508. \doi{10.1214/ejp.v16-967}

\bibitem[{Giles et~al(2016)Giles, Georgiou, and Dettmann}]{giles}
Giles AP, Georgiou O, Dettmann CP (2016) Connectivity of soft random geometric
  graphs over annuli. Journal of Statistical Physics 162:1068--1083.
  \doi{10.1007/s10955-015-1436-1}

\bibitem[{Janson et~al(2010)Janson, {\L}uczak, and Norros}]{janson2010}
Janson S, {\L}uczak T, Norros I (2010) Large cliques in a power-law random
  graph. Journal of Applied Probability 47(4):1124--1135.
  \doi{10.1017/s0021900200007415}

\bibitem[{Janssen et~al(2019)Janssen, Leeuwaarden, and Shneer}]{janssen2019}
Janssen AJEM, Leeuwaarden JSH, Shneer S (2019) Counting cliques and cycles in
  scale-free inhomogeneous random graphs. Journal of Statistical Physics
  175(1):161--184. \doi{10.1007/s10955-019-02248-w}

\bibitem[{Krioukov et~al(2010)Krioukov, Papadopoulos, Kitsak, Vahdat, and
  Bogun{\'a}}]{krioukov2010}
Krioukov D, Papadopoulos F, Kitsak M, et~al (2010) Hyperbolic geometry of
  complex networks. Physical Review E 82(3):036106.
  \doi{10.1103/physreve.82.036106}

\bibitem[{Meester and Roy(1996)}]{MR}
Meester R, Roy R (1996) Continuum Percolation, vol 119. Cambridge University
  Press, Cambridge, \doi{10.1017/cbo9780511895357}

\bibitem[{Michielan and Stegehuis(2022)}]{michielan2021}
Michielan R, Stegehuis C (2022) Cliques in geometric inhomogeneous random
  graphs. Journal of Complex Networks 10(1):1--24. \doi{10.1093/comnet/cnac002}

\bibitem[{Michielan et~al(2022)Michielan, Litvak, and Stegehuis}]{litvak2022}
Michielan R, Litvak N, Stegehuis C (2022) Detecting hyperbolic geometry in
  networks: Why triangles are not enough. Physical Review E 106(5):054303.
  \doi{10.1103/physreve.106.054303}

\bibitem[{Ostilli and Bianconi(2015)}]{ostilli2015}
Ostilli M, Bianconi G (2015) {Statistical mechanics of random geometric graphs:
  Geometry-induced first-order phase transition}. Physical Review E
  91(4):042136. \doi{10.1103/physreve.91.042136}

\bibitem[{Penrose(1991)}]{penrose1991}
Penrose MD (1991) On a continuum percolation model. Advances in Applied
  Probability 23:536--556. \doi{10.2307/1427621}

\bibitem[{Penrose(2003)}]{penrose2003}
Penrose MD (2003) Random Geometric Graphs. Oxford University Press, Oxford,
  \doi{10.1093/acprof:oso/9780198506263.001.0001}

\bibitem[{Penrose(2016)}]{penrose2016}
Penrose MD (2016) Connectivity of soft random geometric graphs. Annals of
  Applied Probability 26(2):986--1028. \doi{10.1214/15-aap1110}

\bibitem[{Penrose(2018)}]{Penrose}
Penrose MD (2018) {Inhomogeneous random graphs, isolated vertices, and Poisson
  approximation}. Journal of Applied Probability 55(1):112--136.
  \doi{10.1017/jpr.2018.9}

\bibitem[{Reitzner and Schulte(2013)}]{Reitzner2013}
Reitzner M, Schulte M (2013) {Central limit theorems for $U$-statistics of
  Poisson point processes}. The Annals of Probability 41(6).
  \doi{10.1214/12-aop817}

\bibitem[{Reitzner et~al(2017)Reitzner, Schulte, and Th{\"a}le}]{rst}
Reitzner M, Schulte M, Th{\"a}le C (2017) {Limit theory for the Gilbert graph}.
  Advances in Applied Mathematics 88:26--61. \doi{10.1016/j.aam.2016.12.006}

\bibitem[{Rousselle and S{\"o}nmez(2021)}]{rs1}
Rousselle A, S{\"o}nmez E (2021) The longest edge in discrete and continuous
  long-range percolation. arXiv preprint arXiv:210911472

\bibitem[{Rousselle and S{\"o}nmez(2023)}]{rs2}
Rousselle A, S{\"o}nmez E (2023) The longest edge of the one-dimensional soft
  random geometric graph with boundaries. Stochastic Models pp 1--19

\end{thebibliography}

\end{document}